\documentclass[11pt,reqno]{amsart}
\usepackage[T1]{fontenc}
\usepackage{graphicx}
\usepackage{color}
\definecolor{MyLinkColor}{rgb}{0,0,0.4}
\newcommand{\R}{{\mathbb R}}
\newcommand{\E}{{\mathcal E}}

\newcommand{\N}{{\mathbb N}}

\newcommand{\wt}{\widetilde}

\newcommand{\ov}{\overline}
\newcommand{\p}{\partial}

\newcommand{\e}{\varepsilon}

\newcommand{\0}{\Omega}

\newtheorem{thm}{Theorem}[section]
\newtheorem{prop}[thm]{Proposition}
\newtheorem{lemma}[thm]{Lemma}

\theoremstyle{remark} 

\numberwithin{equation}{section}

\title[Global weak solutions for a fourth order parabolic system]{Non-negative global weak solutions for a degenerate parabolic system modeling thin films driven by capillarity}
\thanks{}

\author[B.--V. Matioc]{Bogdan--Vasile Matioc}
\address{Institut f{\"u}r Angewandte Mathematik, Leibniz Universit{\"a}t Hannover, Welfengarten~1, 30167 Hannover, Germany.}
\email{matioc@ifam.uni-hannover.de}

\subjclass[2010]{35D30, 35K41, 35K55, 35K65, 35Q35}
\keywords{Thin Film; degenerate parabolic system; non-negative global weak solutions}

\usepackage[colorlinks=true,linkcolor=MyLinkColor,citecolor=MyLinkColor]{hyperref} %dvips

\begin{document}

\begin{abstract}
We prove  global existence of 
 non-negative   weak solutions for a strongly coupled, fourth order  degenerate parabolic system 
governing the motion of two  thin fluid layers in a porous medium when capillarity is the sole driving mechanism.
\end{abstract}

\maketitle

%%% SECTION: INTRO %%%

%%%%%%%%%%%%%%%%%%%%%%%%%%%%%%%%%%%%%%%%
%%%%%%%%%%%%%%%%%%%%%%%%%%%%%%%%%%%%%%%%
\section{Introduction and the main result}\label{S:0}
%%%%%%%%%%%%%%%%%%%%%%%%%%%%%%%%%%%%%%%%
%%%%%%%%%%%%%%%%%%%%%%%%%%%%%%%%%%%%%%%%
In this paper we study the following one-dimensional degenerate system of  equations 
 \begin{equation}\label{eq:S}
\left\{
\begin{array}{llll}
\p_t f=&-\p_x\left[   f\p_x^3\left(Af+Bg\right)\right],\\[1ex]
\p_tg=&-\p_x\left[   g\p_x^3\left(f+g\right)\right],
\end{array}
\right.
{  \quad (t,x)\in (0,\infty)\times (0,L),}
\end{equation}
which models the dynamics of two thin fluid threads in a porous medium in the absence of gravity.
One of the fluids is located in the region bounded from below by  the line $y=0$
and from above by the graph $y=f(t,x)$, while the region occupied by the second fluid is located between the graphs $y=f(t,x)$ and $y=(f+g)(t,x),$ 
 $f$ and $g$ being non-negative functions.
Furthermore, $L$ is a positive real number and the positive  constants $A$ and $B$ have the following physical meaning
\[
A:=\frac{\mu_+}{\mu_-}\frac{\gamma_d+\gamma_w}{\gamma_d}>B:=\frac{\mu_+}{\mu_-}.
\]
We let $\mu_-$ [resp. $\mu_+$] denote the viscosity   of the fluid located below [resp. above],  $\gamma_w$ is the surface tension coefficient at the interface $y=f(t,x)$ between  the wetting phases, while 
$\gamma_d$ is the surface tension coefficient at the interface $y=(f+g)(t,x).$
The system \eqref{P:1} is supplemented by initial conditions
\begin{equation}\label{eq:bc1}
 f(0)=f_0,\qquad g(0)=g_0, { \quad x\in (0,L),}
 \end{equation}
whereby $f_0$ and $g_0$ are assumed to be  known, and we impose no-flux boundary conditions 
\begin{equation}\label{eq:bc2}
 \p_xf=\p_xg=\p^3_xf=\p_x^3 g=0,\qquad x=0, L.
 \end{equation}

The system \eqref{eq:S} has been obtain in \cite{EMM2}, by passing to the small layer thickness in the Muskat problem studied in \cite{EMM1}.    
This is a widely used approach in the study of thin fluid threads because it reduces  complex moving boundary value problems to local problems defined by generally simpler  equations. 
System \eqref{eq:S} is strongly related to the Thin Film equation 
 because if, for instance, $f$ is constantly equal to zero, then $g$ is a solution of the   
Thin Film equation
\begin{equation}\label{eq:TFE}
\partial_t g+\p_x(g^n\p_x^3 g)=0,
\end{equation}
when $n=1.$
We refer to the survey papers \cite{B95, Hul}  where many  aspects  concerning the Thin Film equation are discussed.
It should be noted that similar methods to those in \cite{EMM2} have been used in  \cite{GP} and  \cite{MP} to rigorously show that, in the limit of thin fluid threads, the
 solutions of the moving boundary value problems for Stokes and Hele-Shaw flows converge towards the corresponding  
solutions (determined by the initial data) of the Thin Film equation \eqref{eq:TFE}, with $n=3$ for Stokes  and $n=1$ for the Hele-Shaw flow. 
Compared with the Thin Film equation, system \eqref{eq:S} is more involved because it is strongly coupled, both  equations of 
\eqref{eq:S} containing highest order derivatives of $f$ and $g$, and, furthermore, there are two sources of degeneracy, because both $f$ and $g$ may be equal zero.
Since both  equations of \eqref{eq:S}  have  fourth order, we cannot relay on maximum principles when studying problem \eqref{eq:S}.

Corresponding to  \eqref{eq:S}, we define the following energy functionals 
\[
\E_1(f,g):=\frac{1}{2}\int_0^L  |\p_xf|^2+\frac{B}{A-B}|\p_x(f+g)|^2\, dx,\qquad \E_2(f,g):=\int_0^L\Phi(f)+B\Phi(g)\, dx,
\]
whereby the function $\Phi$  is given by $\Phi(s):=s\ln(s)-s+1$ for all $s\geq0.$
They will play the key role when constructing  the weak solutions for the problem  \eqref{eq:S}-\eqref{eq:bc2}.

Using these two functionals and Galerkin approximations, we prove that the problem \eqref{eq:S}-\eqref{eq:bc2} possesses for non-negative initial data 
non-negative global weak solutions. 
To this end, we regularize first the system \eqref{eq:S} and use the functional $\E_2$ to establish convergence of  certain Galerkin approximations 
towards global
 weak solutions (of the regularized problem) which satisfy similar energy estimates as the classical solutions of \eqref{eq:S}-\eqref{eq:bc2}.
In a second step, we show that weak solutions of the regularized problem converge towards non-negative global weak solutions of the  original system \eqref{eq:S}.
The uniqueness of our weak solutions is left as an open problem (this is still an open problem also for the Thin Film equation  cf. \cite{BF, PGG}). 
We note that it has been only recently shown in \cite{T} (see also \cite{BP1, BP2}), in the context of the Thin Film equation, that the
 non-negative weak solutions found in \cite{BF} converge exponentially fast in $H^1$ towards flat equilibria.
In our case, this is a further open question. 
The second order version of \eqref{eq:S}, when the fluids are driven  only by  gravity and surface tension is neglected,
has been recently investigated in  \cite{ELM} where existence of non-negative  global weak solutions
  which converge exponentially fast in $L_2$ to flat equilibria is established (see also \cite{EMM2}).

In order to state our  main result, we  introduce now the function spaces we work with.
For each $m\in\N,$ we let $H^m:=H^m((0,L))$ be the $L_2-$based Sobolev space and we let $H^m_\Delta$ denote the closed 
subspace of $H^m$ which has  $\{\phi_k\,:\, k\in\N\} $ as its basis.
Herein,
\[
\phi_0:=\sqrt{1/L}\qquad\text{and} \qquad \phi_k:=\sqrt{{2}/{L}}\cos(k\pi x/L), \ k\geq 1,
\]
 are  the normalized eigenvectors of the operator $-\p_x^2:H^2\to L_2$ with zero Neumann boundary conditions.
  To be more precise, $f\in H^m_\Delta$ if and only if the  Fourier series associated to $f$ converges towards $f$ in $H^m.$
It is well-known that $H^1_\Delta=H^1$ 
and, it is not difficult to see that, for $m\geq 4,$ the boundary conditions \eqref{eq:bc2} are satisfied by functions from  this space.

Given $T\in (0,\infty],$ let $Q_T:=(0,T)\times(0,L).$ 
The  main result of this paper is the following theorem.
\begin{thm}\label{T:1} Let $f_0,g_0\in H^1$ be two non-negative functions.
There exist a global weak solution  $(f,g)$ of \eqref{eq:S} with $(f(0),g(0))=(f_0,g_0)$ and having  the following properties:
\begin{itemize}
\item[$(1)$] $f\geq0$ and $g\geq0$ in $(0,T)\times (0,L),$
\item[$(2)$] $f, g\in L_\infty(0,T; H^1)\cap L_2(0,T; H^2_\Delta)\cap C([0,T], C^\alpha([0,L]))$  for some arbitrary
 $\alpha \in(0,1/2)$ and $\sqrt{f}\p_x^3(Af+Bg), \sqrt{g}\p_x^3(f+g)\in L_2(Q_T^+),$
where   
\[
Q_T^+:=\{(t,x)\in Q_T\,:\, (fg)(t,x)>0\},
\]
\item[$(3)$]
\begin{align*}
&\int_{0}^Lf(T) \psi\, dx-\int_{0}^Lf_0 \psi\, dx
+\int_{Q_T} (A\p_x^2f+B\p_x^2g)(\p_xf\p_x\psi +f\p_x^2\psi)\, dxdt=0,\\
&\int_{0}^Lf(T) \psi\, dx-\int_{0}^Lf_0 \psi\, dx
+\int_{Q_T} (\p_x^2f+\p_x^2g)(\p_xg\p_x\psi +g\p_x^2\psi)\, dxdt=0
 \end{align*}
 \end{itemize}
 for all  $T>0$ and $\psi\in H^2_\Delta$.
 Furthermore, the weak solutions satisfy
\begin{align*}
 &(4)\quad \|f(T)\|_{L_1}=\|f_0\|_{L_1}\quad\text{and}\quad \|g(T)\|_{L_1}=\|g_0\|_{L_1},\\
&(5)\quad \E_2(f(T),g(T))+\int_{Q_T}(A-B)|\p_x^2f|^2+B|\p_x^2(f+g)|^2\, dx\, dt \leq \E_2(f_0,g_0)
 \end{align*}
 for  all $T\in(0,\infty),$ and 
\begin{align*}
(6)\quad \E_1(f(T),g(T))+\int_{Q_T^+}f|\p_x^3(Af+Bg)|^2+Bg|\p_x^3(f+g)|^2\, dxdt \leq \E_1(f_0,g_0)
 \end{align*}
for almost all $T\in(0,\infty)$.
\end{thm}

We remark that since $f(t)$ and $g(t)$ belong to $H^2_\Delta$ for almost all $t>0,$ they satisfy homogeneous Neumann boundary conditions at $x=0$ and $x=L$ for  all such $t$.

The outline of the paper is as follows:
 in Section \ref{S:2} we introduce a regularized version of \eqref{eq:S} and use
 Galerkin approximations to find, in the limit, global weak solutions of this regularized problem (see Proposition \ref{P:1}).
Introducing the regularized system allows us on the one
 hand to use the energy functional $\E_1$ when dealing with the Galerkin approximations, and, on the other hand, to 
 control the solutions of the regularized problem when they become negative. 
In Section \ref{S:3} we show, by combining energy estimates for both functionals $\E_1$ and $\E_2$, that 
 the weak solutions of the regularized problem 
 converge towards non-negative global weak solutions of our original problem \eqref{eq:S}-\eqref{eq:bc2}.

%%%%%%%%%%%%%%%%%%%%%%%%%%%%%%%%%%%%%%%%%%%%%%%
%%%%%%%%%%%%%%%%%%%%%%%%%%%%%%%%%%%%%%%%%%%%%%%
\section{The regularized system}\label{S:2}
%%%%%%%%%%%%%%%%%%%%%%%%%%%%%%%%%%%%%%%%
%%%%%%%%%%%%%%%%%%%%%%%%%%%%%%%%%%%%%%%%

In order to prove the Theorem \ref{T:1} we shall  regularize system \eqref{eq:S} and use Galekin approximations to build global weak solutions   for this 
regularized problem.
These solutions are shown later on, in Section \ref{S:3}, to converge towards weak solutions of \eqref{eq:S}.
 To this end, given $\e\in(0,1],$ we define the Lipschitz continuous function $a_\e:\R\to\R$ by the relation
\begin{equation}\label{ses}
a_\e(s):=
\left\{
\begin{array}{lll}
s+\e,&s\geq0,\\
\e,&s<0.
\end{array}
\right.
\end{equation}
Furthermore, we define the convex function $\Phi_\e:\R\to\R$  with
\begin{equation}\label{eq:phi_e}
\Phi_\e(s):=
\left\{
\begin{array}{lll}
(s+\e)\ln(s+\e)-(s+\e)+1,&s\geq0,\\
\displaystyle \frac{ s^2}{2\e}+s\ln(\e)+\e\ln(\e)-\e+1,&s<0.
\end{array}
\right.
\end{equation}
Since we choose $\e\leq1,$ it is easy to  see that $\Phi_\e(s)\geq0$ for all $s\in\R$ and that $\Phi_\e''=1/a_\e.$
With this notation,  we introduce the following regularized version of our original problem \eqref{eq:S}
  \begin{equation}\label{eq:S1}
\left\{
\begin{array}{llll}
\p_t f_\e=&-\p_x\left[   a_\e(f_\e)\p_x^3\left(A f_\e+B g_\e\right)\right],\\[1ex]
\p_tg_\e=&-\p_x\left[   a_\e(g_\e)\p_x^3\left( f_\e+ g_\e\right)\right],
\end{array}
\right.
{  \quad (t,x)\in (0,\infty)\times \0.}
\end{equation}
Of course, this system is coupled with the initial and boundary conditions \eqref{eq:bc1} and \eqref{eq:bc2}. 
Compared to \eqref{eq:S}, the only difference is that we replaced at one place  $f$ and $g$ in \eqref{eq:S} by $a_\e(f)$  and $a_\e(g),$ 
respectively, and  penalize  in this way
 the functions $f_\e, g_\e$  when they take  negative values (see the definition of $a_\e$).
 Furthermore, by choosing the regularization in 
this way, we may still use the functional $\E_1$ to obtain useful estimates for the solutions of \eqref{eq:S1}. 
For the  problem consisting of \eqref{eq:S1} and \eqref{eq:bc1}-\eqref{eq:bc2} we prove the following result.

\begin{prop}\label{P:1} Let $f_0,g_0\in H^1$ be two non-negative functions and $\e\in(0,1].$
There exist globally defined functions $f_\e$ and $g_\e$ with $f_\e(0)=f_0$, $g_\e(0)=g_0$ and having the following properties:
\begin{itemize}
\item[$(i)$] Given  $T>0,$ the functions \[f_\e, g_\e\in L_\infty([0,T], H^1)\cap L_2(0,T; H^3_\Delta)\cap C([0,T], C^\alpha([0,L]))\] 
for some arbitrary $\alpha \in(0,1/2)$.
\item[$(ii)$] For all $T>0$ and $\psi\in H^1$ we have
\begin{align*}
&\int_{0}^Lf_\e(T) \psi\, dx-\int_{0}^Lf_0 \psi\, dx=\int_{Q_T} a_\e(f_\e)\p_x^3(Af_\e+Bg_\e)\p_x\psi\, dxdt,\\
&\int_{0}^Lg_\e(T) \psi\, dx-\int_{0}^Lg_0 \psi\, dx=\int_{Q_T} a_\e(g_\e)\p_x^3(f_\e+g_\e)\p_x\psi\, dxdt.
 \end{align*}
 \item[$(iii)$] The following energy estimates are satisfied:
\begin{align*}
 &(a)\quad \int_0^Lf_\e(T)\, dx=\int_0^L f_0\, dx\quad\text{and}\quad \int_0^L g_\e(T)\,dx=\int_0^L g_0\,dx,\\
 &(b)\quad \int_0^L \Phi_\e(f_\e(T))+B\Phi_\e(g_\e(T))\, dx\\
&\hspace{1.5cm}+\int_{Q_T}(A-B)|\p_x^2f_\e|^2+B|\p_x^2(f_\e+g_\e)|^2\, dx\, dt \\
&\hspace{3cm}\leq \int_0^L \Phi_\e(f_0)+B\Phi_\e(g_0)\, dx\\
 \end{align*}
 
\vspace{-0.5cm}
 for  all $T\in[0,\infty),$ and 
  \begin{align*}
&(c)\quad \E_1(f_\e(T),g_\e(T))\\
&\hspace{1.5cm}+\frac{1}{A-B}\int_{Q_T}a_\e(f_\e)|\p_x^3(Af_\e+Bg_\e)|^2+Ba_\e(g_\e)|\p_x^3(f_\e+g_\e)|^2\, dxdt\\[1ex]
&\hspace{3cm}\leq \E_1(f_0,g_0)\qquad \text{for  almost all $T\in(0,\infty).$}
 \end{align*}
\end{itemize} 
\end{prop}
We will construct the global solutions of \eqref{eq:S1} by using  Galerkin's method.
In a first step we will find, by using the Picard-Lindel\"of theorem, 
Galerkin approximations for the solutions of \eqref{eq:S1} which are defined on a positive time interval.
Using the energy functional $\E_1$, we show then that in fact the approximations are defined globally.
In a second step, we prove that the Galerkin approximation converge  towards global solutions of the regularized system 
which satisfy energy inequalities  for both energy functionals $\E_1$ and $\E_2.$
Though $f_0$ and $g_0$ are non-negative, it is not clear whether $f_\e$ and $g_\e$ preserve this property in time.
However, we will show later on, in Section \ref{S:3}, that, for $\e\to0,$ $f_\e$ and $g_\e$ converge uniformly to non-negative functions.

\vspace{0.2cm}
\subsection{\bf Global existence of  the Galerkin approximations}
$ $ \vspace{0.2cm}

 Given $f_0, g_0$  in $H^1$, the initial conditions of \eqref{eq:S}, we consider their expansions 
\[
f_0=\sum_{k=0}^\infty f_{0k}\phi_k, \qquad g_0=\sum_{k=0}^\infty g_{0k}\phi_k\qquad\text{in $H^1$,}
\]  
and, for each $n\in\N$, the partial sums 
 \[
f_0^n:=\sum_{k=0}^n f_{0k}\phi_k, \qquad g_0^n:=\sum_{k=0}^n g_{0k}\phi_k.
\] 
 We first seek  continuously differentiable  functions  
  \[
f_\e^n:=\sum_{k=0}^n F_\e^k(t)\phi_k, \qquad g_\e^n:=\sum_{k=0}^n G_\e^k(t)\phi_k
\]
 which solve  \eqref{eq:S1} when testing with functions from the vector space $\langle\phi_0,\ldots,\phi_n\rangle,$ and additionally 
 \[
f_\e^n(0)=f_{0}^n,\qquad g_\e^n(0)=g_{0}^n.
\]
By construction, the functions $(f_\e^n,g_\e^n)$ satisfy the boundary conditions \eqref{eq:bc2} and, if we 
test \eqref{eq:S1} with constant functions, it follows at once that necessarily $F_\e^0$ and $G_\e^0$ are constant functions
\begin{equation}\label{eq:00}
F_\e^0(t)=f_{00},\qquad G_\e^0(t)=g_{00},\qquad t\geq0.
\end{equation}
Moreover, the tuple $(\vec F_\e^n,\vec G_\e^n):=(F_\e^1,\ldots,F_\e^n,G_\e^1,\ldots,G_\e^n)$
is the solution of the initial value problem
\begin{equation}\label{eq:IVP}
(\vec F_\e^n,\vec G_\e^n)'=\Psi(\vec F_\e^n,\vec G_\e^n),\qquad (\vec F_\e^n,\vec G_\e^n)(0)=(f_{01},\ldots,f_{0n},g_{01},\ldots,g_{0n}),
\end{equation}
 where $\Psi:=(\Psi_1,\Psi_2):\R^{2n}\to\R^{2n}$
 is  given by
\begin{align*}
\Psi_{1,j}(x,y)=&\sum_{k=1}^n (Ax_k+By_k)\int_0^La_\e\left(f_{00}\phi_0+\sum_{l=1}^nx_l\phi_l\right)\p_x^3\phi_k\p_x\phi_j\, dx\\
\Psi_{2,j}(x,y)=&\sum_{k=1}^n(x_k+y_k)\int_0^La_\e\left(g_{00}\phi_0+\sum_{l=1}^ny_l\phi_l\right)\p_x^3\phi_k\p_x\phi_j\, dx,
 \end{align*}
 for all $x,y\in\R^n.$
 Since $a_\e$ is Lipschitz continuous, we deduce that $\Psi$ is locally Lipschitz continuous on $\R^{2n},$ 
and therefore problem \eqref{eq:IVP} possesses a unique solution $(\vec F_\e^n,\vec G_\e^n)$ 
defined on a maximal interval $[0,T_\e^n)$.
 In order to prove that the solution is global, that is $T_\e^n=\infty$ for all $\e\in(0,1]$ and $n\in\N,$ we make use of the energy functional $\E_1$.
 Indeed, since $\p_x^2 f_\e^n,\p_x^2 g_\e^n\in\langle\phi_0,\ldots,\phi_n\rangle,$  we may use them as test functions for \eqref{eq:S1}. 
Integrating by parts,  we then get the following relation
 \begin{equation}\label{E:31}
\begin{aligned}
 &\frac{d}{dt}\E_1(f_\e^n,g_\e^n)\\
&\hspace{0.5cm}=\frac{1}{A-B}\int_0^L  A\p_x f^n_\e\p_t(\p_xf^n_\e)+B\p_x f^n_\e\p_t(\p_xg^n_\e)\\
&\hspace{3cm}+B\p_x g^n_\e\p_t(\p_xf^n_\e)+B\p_x g^n_\e\p_t(\p_xg^n_\e)\,dx\\
 &\hspace{0.5cm}=-\frac{1}{A-B}\int_0^L  A\p_x^2 f^n_\e\p_tf^n_\e+B\left[\p_x^2 f^n_\e\p_tg^n_\e\, dx+\p_x^2 g^n_\e\p_tg^n_\e+\p_x^2 g^n_\e\p_tf^n_\e\right]\,dx\\
 &\hspace{0.5cm}=-\frac{1}{A-B}\int_0^L \left[ Aa_\e(f_\e^n)\p_x^3 f^n_\e\p_x^3 (Af^n_\e+Bg_\e^n)+Ba_\e(g_\e^n)\p_x^3 f^n_\e\p_x^3 (f^n_\e+g_\e^n)\right.\\
&\hspace{3cm}+\left.Ba_\e(f_\e^n)\p_x^3 g^n_\e\p_x^3 (Af^n_\e+Bg_\e^n)+Ba_\e(g_\e^n)\p_x^3 g^n_\e\p_x^3 (f^n_\e+g_\e^n)\right]\, dx,
 \end{aligned}
 \end{equation}
 and taking into account that $\E_1(f_0^n,g_0^n)\leq \E_1(f_0,g_0)$ for all $n\in\N$, we  find after  integrating with respect to time that
\begin{align}
&\E_1(f_\e^n(T),g_\e^n(T))+\frac{1}{A-B}\int_{Q_T}a_\e(f^n_\e)|\p_x^3(Af^n_\e+Bg^n_\e)|^2\nonumber\\
&\hspace{5.2cm}+Ba_\e(g^n_\e)|\p_x^3(f^n_\e+g^n_\e)|^2\, dx\, dt\leq \E_1(f_0,g_0)\label{eq:EE}
 \end{align}
 for all $T>0$.
 Whence,  there exists a positive  constant $C$, which is independent of time, such that
$|(\vec F_\e^n(T),\vec G_\e^n(T))|<C$ for all $T<T_\e^n.$ 
Together with \eqref{eq:00}, we conclude that for each $n\in\N$ and $\e\in(0,1]$, the Galerkin approximations $(f_\e^n,g_\e^n)$ are defined globally.

 \vspace{0.2cm}
\subsection{\bf Convergence of the Galerkin approximations}
$ $ \vspace{0.2cm}

 Let $T>0$ and $\e\in(0,1]$ be fixed. 
 From the energy estimate \eqref{eq:EE}  we deduce that
 \begin{align}\label{E:1}
 \p_xf_\e^n,\  \p_xg_\e^n \quad \text{are  bounded in $L_\infty(0,T; L_2)$,}\\
\label{E:2}
 \sqrt{a_\e(f^n_\e)}\p_x^3(Af^n_\e+Bg^n_\e),\ \sqrt{a_\e(g^n_\e)}\p_x^3(f^n_\e+g^n_\e)\quad \text{are  bounded in $L_2(Q_T)$,}
 \end{align}
 uniformly in $n\in\N$ and $\e\in(0,1]$.
In view of $a_\e\geq\e$ and  $A>B$, we obtain from \eqref{E:2} that 
 \begin{align}
 \label{E:3}
 \p_x^3f_\e^n, \ \p_x^3g_\e^n &\quad \text{are  bounded in $L_2(Q_T)$,}
 \end{align}
 uniformly in $n\in\N$. 
Furthermore, by virtue  of \eqref{eq:00}, we see that the mass of both  fluids is preserved  by the Galerkin approximations
 \begin{align}\label{E:4}
 \int_0^Lf_\e^n(t)\, dx=\int_0^L f_0\, dx\quad\text{and}\quad \int_0^L g_\e^n(t)\,dx=\int_0^L g_0\,dx\quad\text{for all $t\in[0,T].$}
 \end{align}
 
 Invoking now \eqref{E:1}, \eqref{E:4},  and the  Poincar\'e-Wirtinger inequality  we conclude that in fact
 \begin{align}\label{E:5}
 f_\e^n,\  g_\e^n &\quad \text{are  bounded in $ L_\infty(0,T; H^1)$ uniformly in $\e\in(0,1]$ and $n\in\N$,}
\end{align}
while,  owing to  \eqref{E:3} and  \eqref{E:4}, the same inequality   implies
\begin{align}\label{E:6}
 f_\e^n,\  g_\e^n &\quad \text{are  bounded in $ L_2(0,T; H^3)$ uniformly in $n$.}
\end{align}
 We consider now the partial derivatives with respect to time, and observe that the first equation of \eqref{eq:S1}  can be written in the more compact form
 $\p_tf_\e^n=-\p_xH^n_\e$
 where, by \eqref{E:2},  \eqref{E:5}, and using the embedding $H^1\hookrightarrow L_\infty,$ the right-hand side 
$H^n_\e:= a(f^\e_n)(A\p_x^3f_\e^n+B\p_x^3g_\e^n)$ is  bounded in $L_2(Q_T)$ uniformly in $\e$ and $n$.
 Therefore, given $\zeta\in H^1,$ we set $$\zeta^n:=\sum_{k=0}^n(\zeta|\phi_k)\phi_k$$ and, using  integration by parts, obtain
 \begin{align*}|(\p_tf^n_\e(t)|\zeta)|=&|(\p_tf^n_\e(t)|\zeta^n)|\\
=&|( H^n_\e|\p_x\zeta_n)|\\
\leq& \|H^n_\e\|_{L_2(Q_T)}\|\zeta^n\|_{H^1}\\
\leq &\|H^n_\e\|_{L_2(Q_T)}\|\zeta\|_{H^1}.\end{align*}
 This means that
  \begin{align}\label{E:7}
 \p_tf_\e^n,\  \p_tg_\e^n &\quad \text{are  bounded in $L_2(0,T;  (H^1)')$ uniformly in $\e$ and $n$.}
\end{align}
 Gathering \eqref{E:5}-\eqref{E:7}, we obtain from Corollary 4 in \cite{JS}, by making also use of the 
embeddings \[\text{$H^1\overset{comp.}\hookrightarrow C^\alpha([0,L])\hookrightarrow (H^1)'$ \, and \,
  $H^3\overset{comp.}\hookrightarrow C^{2+\alpha}([0,L])\hookrightarrow (H^1)'$}\] for $\alpha\in[0,1/2)$, that 
 \[
f_\e^n,\  g_\e^n \quad \text{are relatively compact in $C([0,T], C^\alpha([0,L]))\cap L_2(0,T;C^{2+\alpha}([0,L]))$.}
\]
Whence, for each $\e\in(0,1],$ there exist functions 
\[f_\e, g_\e\in C([0,T], C^\alpha([0,L]))\cap L_2(0,T;C^{2+\alpha}([0,L]))\] 
and subsequences of $(f_\e^n)$ and $(g_\e^n)$ (which we denote again by  $(f_\e^n)$ and $(g_\e^n)$) 
 such that
 \begin{equation}\label{F:1}
f_\e^n\to f_\e\quad\text{and}\quad  g_\e^n\to g_\e\quad \text{ in $C([0,T], C^\alpha([0,L]))\cap L_2(0,T;C^{2+\alpha}([0,L]))$.}
\end{equation}
Moreover,    we deduce from \eqref{E:6} that
 \begin{equation}\label{F:2}
\p_x^p f_{\e}^n\rightharpoonup  \p_x^pf_\e\quad\text{and}\quad  \p_x^p g_{\e}^n\rightharpoonup  \p_x^pg_\e\quad \text{ in $L_2(Q_T)$ for $p=1,2,3,$} 
\end{equation}
and therefore   $f_\e,  g_\e\in L_2([0,T], H^3)$.
Additionally, since  $f_\e^n(t), g_\e^n(t)\in H^3_\Delta$, we get, by virtue of \eqref{F:1}, that   $\p_xf_\e=\p_x g_\e=0$   at $x=0,L$ for 
almost all $t\in[0,T],$
which yields $f_\e,  g_\e\in L_2([0,T], H^3_\Delta)$. \vspace{0.2cm}

\subsection{\bf Proof of Proposition \ref{P:1}}
$ $ \vspace{0.2cm}

 First of all,   $f_\e^n(0)=f_0^n$ for all $n\in\N$ and since
$f_0\in H^1$ 
we conclude that    $f_\e(0)=f_0 $ for all $\e\in(0,1].$
Similarly, we have $g_\e(0)=g_0$ for all $\e\in(0,1].$
Furthermore, it is clear from \eqref{E:4} and \eqref{F:1} that the weak solutions $(f_\e,g_\e)$ satisfy the relation $(iii)(a)$  of Proposition \ref{P:1}.

We pass now to the limit in the energy estimate \eqref{eq:EE}.
By virtue of \eqref{E:2}, \eqref{F:1}, and \eqref{F:2} we have
\[
\begin{array}{lll}
&\sqrt{a_\e(f^n_\e)}\p_x^3(Af^n_\e+Bg^n_\e)\rightharpoonup \sqrt{a_\e(f_\e)}\p_x^3(Af_\e+Bg_\e),\\[1ex]
&\sqrt{a_\e(g^n_\e)}\p_x^3(f^n_\e+g^n_\e)\rightharpoonup \sqrt{a_\e(g_\e)}\p_x^3(f_\e+g_\e)
\end{array}\quad\text{ in $L_2(Q_T)$.}
\]
Furthermore, by \eqref{F:1} we know that $f_\e^n(t)\to f_\e(t)$ in $H^1$ for almost all $t\in[0,T],$ so that, by
passing to the limit $n\to\infty$ in \eqref{eq:EE}, we obtain the estimate $(iii)(c)$ of Proposition \ref{P:1}.

Claim $(i)$ of Proposition \ref{P:1} is now a simple consequence of the assertions  $(iii)(a)$ and $(iii)(c)$ of the same proposition.

We  now prove the assertion $(ii)$ of Proposition \ref{P:1}.
To this end, we pick an arbitrary function $\psi\in H^1$ and, testing \eqref{eq:S1} with $\psi^n:=\sum_{k=0}^n(\psi|\phi_k)\phi_k,$ we obtain the following relations
 \begin{align*}
&\int_0^Lf^n_\e(T) \psi^n\, dx-\int_0^Lf_{0n} \psi^n\, dx=\int_{Q_T} a_\e(f_\e^n)\p_x^3(Af_\e^n+Bg_\e^n)\p_x\psi^n\, dxdt,\\
&\int_0^Lg^n_\e(T) \psi^n\, dx-\int_0^Lg_{0n} \psi^n\, dx=\int_{Q_T} a_\e(g_\e^n)\p_x^3(f_\e^n+g_\e^n)\p_x\psi^n\, dxdt.
 \end{align*}
 Invoking \eqref{E:2} and \eqref{E:5}, we see that $a_\e(f_\e^n)\p_x^3(Af_\e^n+Bg_\e^n)$ and $a_\e(g_\e^n)\p_x^3(f_\e^n+g_\e^n)$ are bounded 
in $L_2(Q_T)$ uniformly in $\e$ and $n$.
Using    \eqref{F:1} and \eqref{F:2}, 
we may even identify their weak limit 
\begin{equation}\label{eq:QE}
\begin{aligned}
&a_\e(f^n_\e)\p_x^3(Af^n_\e+Bg^n_\e)\rightharpoonup a_\e(f_\e)\p_x^3(Af_\e+Bg_\e),\\[1ex]
&a_\e(g^n_\e)\p_x^3(f^n_\e+g^n_\e)\rightharpoonup a_\e(g_\e)\p_x^3(f_\e+g_\e)
\end{aligned}\quad\text{ in $L_2(Q_T)$,}
\end{equation}
and the assertion  $(ii)$ of Proposition \ref{P:1} follows from the previous identities when letting $n\to\infty$.

We end this paragraph with the proof of the estimate $(iii)(b)$ of Proposition \ref{P:1}. 
Let us observe that  $\Phi_\e'(f_\e^n(t))$  and  $\Phi_\e'(f_\e(t))$  belong to $H^1$ for almost all $t\in[0,T],$
meaning that
\begin{equation}\label{eq:FA}
\Phi_\e'(f_\e^n(t))=\sum_{k=0}^\infty(\Phi_\e'(f_\e^n(t))| \phi_k)\phi_k, \qquad
 \Phi_\e'(f_\e(t))=\sum_{k=0}^\infty(\Phi_\e'(f_\e(t))| \phi_k)\phi_k\quad\text{in $H^1$}
\end{equation}
 for almost all $t\in[0,T].$
 Of course, \eqref{eq:FA} is  also valid  when replacing $f$ by $g$.
In view of \eqref{eq:FA}, we obtain the following relations
  \begin{equation}\label{E:32}
\begin{aligned}
\frac{d}{dt}\int_0^L \Phi_\e(f_\e^n)+B\Phi_\e(g_\e^n)\, dx=&\int_0^L \Phi_\e'(f_\e^n)\p_tf_\e^n+B\Phi_\e'(g_\e^n)\p_tg_\e^n\, dx\\
=&\int_0^L a_\e(f_\e^n)\p_x^3(Af_\e^n+Bg_\e^n)\sum_{k=0}^n(\Phi_\e'(f_\e^n)|\phi_k)\p_x\phi_k \\
&+Ba_\e(g_\e^n)\p_x^3(f_\e^n+g_\e^n)\sum_{k=0}^n(\Phi_\e'(g_\e^n)|\phi_k)\p_x\phi_k \, dx,
\end{aligned}
\end{equation}
and, integrating  with respect to time,  we arrive at
\begin{equation}\label{eq:qe}
\begin{aligned}
&\int_0^L \Phi_\e(f_\e^n(T))+B\Phi_\e(g_\e^n(T))\, dx\\
&\hspace{1.5cm}=\int_{Q_T} \left[a_\e(f_\e^n)(A\p_x^3f_\e^n+B\p_x^3g_\e^n)\sum_{k=0}^n(\Phi_\e'(f_\e^n)|\phi_k)\p_x\phi_k  \right.\\
&\hspace{3cm}+\left.Ba_\e(g_\e^n)(\p_x^3f_\e^n+\p_x^3g_\e^n)\sum_{k=0}^n(\Phi_\e'(g_\e^n)|\phi_k)\p_x\phi_k \right] dx dt\\
&\hspace{2cm}+\int_0^L \Phi_\e(f_0^n)+B\Phi_\e(g_0^n)\, dx.
\end{aligned}
\end{equation}
In order to pass to the limit $n\to\infty$ in relation \eqref{eq:qe} we have to determine what happens  with the 
two integrals on the right-hand side of \eqref{eq:qe}.
Using \eqref{eq:FA}, we have
\begin{align*}
&\left\|\sum_{k=0}^n(\Phi_\e'(f_\e^n)|\phi_k)\p_x\phi_k-\Phi''_\e(f_\e)\p_xf_\e\right\|^2_{L_2(Q_T)}\\
&\hspace{1.5cm}\leq2\left\|\sum_{k=0}^n(\Phi_\e'(f_\e^n)-\Phi_\e'(f_\e)|\phi_k)\p_x\phi_k\right\|^2_{L_2(Q_T)}\\
&\hspace{2cm}+2\left\|\sum_{k=0}^n(\Phi_\e'(f_\e)|\phi_k)\p_x\phi_k-\Phi''_\e(f_\e)\p_xf_\e\right\|^2_{L_2(Q_T)}.
\end{align*}
Taking into account that the first sum on the right-hand side of the latter inequality is the truncation of the Fourier series of 
$\Phi''_\e(f_\e^n)\p_xf_\e^n-\Phi''_\e(f_\e)\p_xf_\e$, cf. \eqref{eq:FA}, its norm may be estimated as follows 
\begin{align*}
 \left\|\sum_{k=0}^n(\Phi_\e'(f_\e^n)-\Phi_\e'(f_\e)|\phi_k)\p_x\phi_k\right\|^2_{L_2(Q_T)}
\leq&\|\Phi''_\e(f_\e^n)\p_xf_\e^n-\Phi''_\e(f_\e)\p_xf_\e\|^2_{L_2(Q_T)}\\
 \leq&2\|\Phi''_\e(f_\e^n)-\Phi''_\e(f_\e)\|_{L_\infty(Q_T)}^2\|\p_xf_\e\|^2_{L_2(Q_T)}\\
&+2\|\Phi''_\e(f_\e^n)\|_{L_\infty(Q_T)}^2\|\p_xf_\e^n-\p_xf_\e\|^2_{L_2(Q_T)}\\
 \leq&2\e^{-4}\|f_\e^n-f_\e\|_{L_\infty(Q_T)}^2\|\p_xf_\e\|^2_{L_2(Q_T)}\\
&+2\e^{-2}\|\p_xf_\e^n-\p_xf_\e\|^2_{L_2(Q_T)}.
\end{align*}
We note that the last inequality has been obtained by using the fact that $\Phi''_\e $ is Lipschitz continuous with Lipschitz constant $\e^{-2}$ and
 $0\leq \Phi_\e''\leq \e^{-1}$, 
properties which readily follow  from  \eqref{ses},
\eqref{eq:phi_e}, and the relation $\Phi_\e''=1/a_\e$.  
Invoking \eqref{F:1}, we resume our calculation with
\begin{equation}\label{A0}
\left\|\sum_{k=0}^n(\Phi_\e'(f_\e^n)-\Phi_\e'(f_\e)|\phi_k)\p_x\phi_k\right\|^2_{L_2(Q_T)}\to_{n\to\infty}0.
\end{equation}
Concerning  the second term, we obtain from \eqref{eq:FA} that
\begin{equation*}
\left\|\sum_{k=0}^n(\Phi_\e'(f_\e)|\phi_k)\p_x\phi_k-\Phi''_\e(f_\e)\p_xf_\e\right\|_{L_2}
=\left\|\sum_{k=n+1}^\infty(\Phi_\e'(f_\e)|\phi_k)\p_x\phi_k\right\|_{L_2}\searrow _{n\to\infty}0
\end{equation*}
for almost all $t\in[0,T],$ and Lebesgue's dominated convergence theorem yields 
\begin{equation}\label{A3}
\left\|\sum_{k=0}^n(\Phi_\e'(f_\e^n)|\phi_k)\p_x\phi_k-\Phi''_\e(f_\e)\p_xf_\e\right\|_{L_2(Q_T)}\to_{n\to\infty}0.
\end{equation}
Gathering \eqref{A0} and \eqref{A3}, we conclude that
\begin{equation}\label{A4}
\sum_{k=0}^n(\Phi_\e'(f_\e^n)|\phi_k)\p_x\phi_k\to_{n\to \infty}\Phi''_\e(f_\e)\p_xf_\e\qquad\text{in $L_2(Q_T)$.}
\end{equation}
Clearly, \eqref{A4} remains true if we replace $f$ by $g.$
We sum   \eqref{eq:QE}, \eqref{A4}, use \eqref{F:1} and the fact that both $f_0$ and $g_0$ 
are non-negative to obtain from \eqref{eq:qe}, when letting $n\to\infty$, the desired assertion $(iii)(b)$ of Proposition \ref{P:1}.

\section{The proof of Theorem \ref{T:1}}\label{S:3}
We shall use the global weak solutions $(f_\e,g_\e)$ of the regularized problem \eqref{eq:S1} to find,
 in the limit $\e\to 0$, global weak solutions of our original system \eqref{eq:S}.
The key role is now played   by the second energy functional $\E_2$, which will be used to prove that the weak solutions we obtain are non-negative
 and to identify in $L_2(0,T;H^2)$ a weak limit of the global solutions of \eqref{eq:S1}.
Using  integration by parts, we may eliminate then from the right-hand side of $(ii)$ Proposition \ref{P:1} the
 third order derivatives of $f_\e$ and $g_\e$, for which we don't have any kind of uniform bounds, and obtain in the limit $\e\to0$ 
the assertion $(3)$ of Theorem \ref{T:1}. 

To do so, we collect first some estimates for the family $(f_\e,g_\e)$ which have been already established in Section \ref{S:2}.
We have to  pay attention because some of the estimates proven before are uniform only with respect to $n$, and of no use in this final part. 
Invoking \eqref{E:2} and \eqref{E:5},   we deduce the uniform boundedness of
\begin{align}\label{B:1}
 f_\e,\  g_\e \quad \text{  in $L_\infty(0,T; H^1)$,}\\
 \label{B:3}
  \sqrt{a_\e(f_\e)}\p_x^3(Af_\e+Bg_\e), \ \sqrt{a_\e(g_\e)}\p_x^3(f_\e+g_\e)\quad \text{ in $L_2(Q_T)$,}
 \end{align}
 while, by virtue of Proposition $(iii) (a)$ and $(b),$
 \begin{align}
 \label{B:2}
 &\p_x^2f_\e, \ \p_x^2 g_\e \quad \text{are uniformly bounded in $L_2(Q_T)$,}\\
 \label{B:5}
  &\int_0^Lf_\e(T)\, dx=\int_0^L f_0\, dx,\,\quad  \int_0^L g_\e(T)\,dx=\int_0^L g_0\,dx,
 \end{align}
 for all $T>0.$
Lastly, we observe that the estimates   \eqref{E:5} and \eqref{E:7} are both uniform with respect to $\e\in(0,1]$ and $n\in\N.$
This implies that the families $\{f_\e^n\,:\, \e\in(0,1], \ n\in\N\}$ and  $\{g_\e^n\,:\, \e\in(0,1], \ n\in\N\}$ are both relatively 
compact in $C([0,T], C^\alpha([0,L]))$, if $\alpha\in[0,1/2),$ and therefore 
\begin{align}
 \label{B:4}
 &(f_\e), (g_\e) \quad \text{are  relatively compact in $C([0,T], C^\alpha([0,L]))$.}
 \end{align}
Consequently,  there exist subsequences $(f_{\e_k})$ and $(g_{\e_k})$ and functions $f, g$ such that
\begin{equation}\label{B:7}
f_{\e_k}\to f\quad\text{and}\quad  g_{\e_k}\to g\quad \text{ in $C([0,T], C^\alpha([0,L]))$,}
\end{equation}
while, owing to  \eqref{B:1}, \eqref{B:2}, we conclude that $f_\e, g_\e$ are bounded in $ L_2(0,T;H^2),$ 
which ensures, after possibly extracting further subsequences, weak convergence in $L_2(Q_T)$ of the spatial derivatives  up to order 2
\begin{equation}\label{B:8}
\p_x^p f_{\e_k}\rightharpoonup  \p_x^pf\quad\text{and}\quad  \p_x^p g_{\e_k}\rightharpoonup  \p_x^pg\quad \text{ in $L_2(Q_T)$ for $p=1,2.$}
\end{equation}
Recalling Proposition \ref{P:1} $(i)$ and \eqref{B:7}, we deduce that  $f, g \in L_2(0,T;H^2_\Delta)$ for all $T>0.$
Moreover, the sequences $(f_{\e_k})$ and $(g_{\e_k})$ converge strongly towards  $f$ and $ g$, respectively, in a different norm than in \eqref{B:7}.

\begin{lemma}\label{L:2}  Given $T>0,$ we have:
\begin{equation}\label{e:BB}
f_{\e_k}\to f\qquad\text{and}\qquad g_{\e_k}\to g \qquad\text {in $L_4(0,T;H^1).$}
\end{equation}
\end{lemma}
\begin{proof} We prove only the assertion for $f$.
Since $f(t)$ and $ f_{\e_k}(t)$ belong to $H^2_\Delta$ for almost  all $t\in[0,T],$ we conclude that their first order derivatives at $0$ and $L$ must vanish.
Whence, using integration by parts, we get
\begin{align*}
&\int_{0}^T\left(\int_0^L|\p_x(f_{\e_k}-f)|^2\, dx\right)^2dt=\int_{0}^T\left(\int_0^L\p_x^2(f_{\e_k}-f)(f_{\e_k}-f)\, dx\right)^2dt\\
&\leq \int_0^T\|\p_x^2(f_{\e_k}-f)\|_{L_2}^2\|f_{\e_k}-f\|_{L_2}^2\, dt\leq L\|f_{\e_k}-f\|_{L_\infty(Q_T)}^2\|\p_x^2(f_{\e_k}-f)\|_{L_2(Q_T)}^2,
\end{align*}
and, together with \eqref{B:2} and \eqref{B:7}, we get the desired conclusion.
\end{proof}
Particularly, we obtain from \eqref{e:BB}, that $f_{\e_k}(t)\to f(t)$ and $g_{\e_k}(t)\to g(t)$ in $H^1$ for almost all $t\in[0,T]$, and 
together with the estimate \eqref{B:1} we conclude that
$f, g \in L_\infty(0,T;H^1).$
Furthermore, $f_\e(0)=f_0$ and $g_\e(0)=g_0$ for all $\e\in(0,1]$, so that  \eqref{B:7} yields $f(0)=f_0$ and $g(0)=g_0 $.
The estimate $(4)$ of Theorem \ref{T:1} follows by combining \eqref{B:7},  the assertion $(iii)(a)$ of Proposition \ref{P:1}, and Lemma \ref{L:1} below.

We use now the  energy estimate $(iii)(b) $ of Proposition \ref{P:1},  to establish the assertion $(1)$ of our main result Theorem \ref{T:1}.
\begin{lemma}\label{L:1} The functions $f$ and $g$ found above are non-negative.
\end{lemma}
\begin{proof} Assume that there exists $(T,x_0)\in Q_\infty$ such that $f(T,x_0)<0$.
Since by \eqref{B:7}  $f_{\e_k}\to f$ in  $C(\ov Q_T)$, we conclude that there exists a  constant $\delta>0 $ and $k_0\in\N$ with the property that
$f_{\e_k}(T,x)<-\delta$ for all $x\in[0,L]$ with $|x-x_0|<\delta$ and all $k\geq k_0.$  
We then infer  from \eqref{eq:phi_e} that
\[
\Phi_{\e_k}(f_{\e_k}(T,x))=\frac{ f_{\e_k}^2(T,x)}{2\e_k}+f_{\e_k}(T,x)\ln(\e_k)+\e_k\ln(\e_k)-\e_k+1\geq\frac{\delta^2}{2\e_k}
\]
for all $x$ and $k$ as above.
This contradicts the assertion $(iii)(b)$ of Proposition \ref{P:1}. 
Clearly, the argument is true when replacing $f$ by $g$, and this proves the claim.
\end{proof} 

In order to deduce the energy estimate Theorem \ref{T:1} $(5)$, we recall \eqref{eq:phi_e} and notice that, 
for all $k\in\N,$ we have $\Phi_{\e_k}(f_{\e_k})\geq \wt\Phi_{\e_k} (f_{\e_k}), $ where
\begin{equation*}
\wt\Phi_{\e_k} (s):=
\left\{
\begin{array}{lll}
(s+\e_k)\ln(s+\e_k)-(s+\e_k)+1,&s\geq0,\\
\e_k\ln(\e_k)-\e_k+1,&s<0.
\end{array}
\right.
\end{equation*}
Given $t\in[0,T],$ the sequence $(\wt\Phi_{\e_k} (f_{\e_k}(t)))$ is bounded in $C([0,L])$ 
and $\wt\Phi_{\e_k} (f_{\e_k}(t))\to \Phi(f(t))$ pointwise on $[0,L]$.
 Lebesgue's dominated convergence implies then
 \begin{equation}
\liminf_{k\to\infty}\int_{0}^L\Phi_{\e_k}(f_{\e_k}(T))\, dx\geq \liminf_{k\to\infty}\int_0^L\wt\Phi_{\e_k} (f_{\e_k}(T))\, dx=\int_0^L\Phi(f(T))\, dx.
\end{equation}
Of course, the relation still remains true when replacing $f$ by $g$.
By virtue of \eqref{B:8}, we may pass to $\liminf_{k\to\infty}$ in relation $(iii)(b)$ of 
Proposition \ref{P:1}, and  obtain in this way the desired energy estimate $(5)$ of Theorem \ref{T:1}.  

To deal with the energy estimate $(6)$ of Theorem \ref{T:1},
we observe first that  for all $k\in\N$
 \[
|a_{\e_k}(f_{\e_k})-f|\leq {\e_k}+|f_{\e_k}-f|,
\]
meaning, by \eqref{B:7}, that 
\begin{equation}\label{B:9}
 a_{\e_k}(f_{\e_k})\to  f\quad\text{and}\quad a_{\e_k}(g_{\e_k})\to  g\quad \text{ in $C(\ov Q_T)$.}
 \end{equation}
For every positive integer $m$, we introduce now the set
\[
Q_T^m:=\{(t,x)\in Q_T\,:\, f(t,x)>m^{-1} \ \text{and} \ g(t,x)>m^{-1}\},
\]
where we may control,  by virtue of the estimate $(iii)(c)$ of Proposition \ref{P:1} and \eqref{B:9}, 
the third order derivatives of both $f_{\e_k}$ and $g_{\e_k}:$
\begin{equation}\label{E:14}
(\p_x^3f_{\e_k}), (\p_x^3g_{\e_k}) \quad \text{are  uniformly bounded in $L_2(Q_T^m)$.}
\end{equation}
Taking into account that $Q_T^+=\cup_{m} Q_T^m,$ we may assume, after possibly extracting a further subsequence, that
\begin{equation*}
\p_x^3f_{\e_k}\rightharpoonup \p_x^3 f,\quad  \p_x^3g_{\e_k}\rightharpoonup \p_x^3g \quad \text{ in $L_2(Q_T^m)$}
\end{equation*}
for all $m\in\N,$ which, together with \eqref{B:9}, implies  
\begin{equation}\label{F:3}
\begin{aligned}
&\sqrt{a_{\e_k}(f_{\e_k})}\p_x^3(Af_{\e_k}+Bg_{\e_k})\rightharpoonup \sqrt{f}\p_x^3(Af+Bg), \\
& \sqrt{a_{\e_k}(g_{\e_k})}\p_x^3(f_{\e_k}+g_{\e_k})\rightharpoonup \sqrt{g}\p_x^3(f+g)
\end{aligned}
 \qquad \text{ in $L_1(Q_T^m)$}.
 \end{equation}
 In fact, by virtue of \eqref{B:3}, the weak convergence in \eqref{F:3} takes place in $L_2(Q_T^m).$
Recalling Proposition \ref{P:1} $(iii)(c)$ and Lemma \ref{L:2},  for $k\to\infty$, we obtain the  
 desired estimates $(2)$ and $(6)$ of Theorem \ref{T:1}. 

In order to complete the proof of Theorem \ref{T:1}, we are left to prove the relations $(3)$.
To this end, we pick $\psi\in H^2_\Delta.$
Since $a_{\e_k}$ is Lipschitz continuous, we obtain from Proposition \ref{P:1} $(i)$ that $a_{\e_k}(f_{\e_k}(t))\in H^1$  for almost all $t\in(0,T)$  and 
\[
\p_x(a_{\e_k}(f_{\e_k}))(t,x)=\chi_{(0,\infty)}(f_{\e_k})\p_xf_{\e_k},\qquad \text{a.e. in $Q_T,$}
\]
whereby $\chi_{(0,\infty)}$ denotes the characteristic function of the interval $(0,\infty).$
Integrating by parts in the first relation of Proposition \ref{P:1} $(ii)$ , we arrive at
\begin{equation}\label{B:B1}
\int_{0}^Lf_{\e_k}(T) \psi\, dx-\int_{0}^Lf_0 \psi\, dx=\int_{0}^T\int_0^L a_{\e_k}(f_{\e_k})\p_x^3(Af_{\e_k}+Bg_{\e_k})\p_x\psi\, dxdt=I_{1,k}+I_{2,k}\\
\end{equation}
where
\begin{align*}
&I_{1,k}:=-\int_{Q_T} \p_x(a_{\e_k}(f_{\e_k}))\p_x^2(Af_{\e_k}+Bg_{\e_k})\p_x\psi\, dxdt,\\ 
&I_{2,k}:=-\int_{Q_T} a_{\e_k}(f_{\e_k})\p_x^2(Af_{\e_k}+Bg_{\e_k})\p_x^2\psi\, dxdt.
\end{align*}
We note the use of $\psi\in H^2_\Delta$ to eliminate the boundary terms in \eqref{B:B1}   due to $\p_x\psi(0)=\p_x\psi(L)=0$.

Combining \eqref{B:8} and \eqref{B:9}, we obtain for  $k\to\infty$  that
 \begin{equation}\label{B:10}
 I_{2,k}\to-\int_{Q_T} f(A\p_x^2f+B\p_x^2g)\p_x^2\psi\, dxdt.
 \end{equation}
We consider now the integral $I_{1,k},$ and notice that in order to show the relation
\begin{align}\label{B:100}
I_{1,k}\to-\int_{Q_T} \p_xf(A\p_x^2f+B\p_x^2g)\p_x\psi\, dxdt
\end{align}
it suffices to prove that 
\begin{align}\label{eq:gau}
\p_x(a_{\e_k}(f_{\e_k}))\to\p_xf\qquad\text{in $L_2(Q_T)$.} 
\end{align}
To this end, we write $\p_x(a_{\e_k}(f_{\e_k}))-\p_xf=\left(\p_x(a_{\e_k}(f_{\e_k}))-\p_xf_{\e_k}\right)+\left(\p_xf_{\e_k}-\p_xf\right),$ and 
conclude from Lemma \ref{L:2} that
$\left(\p_xf_{\e_k}-\p_xf\right)\to0$ in $L_2(Q_T).$
Furthermore, the first term may be written as $\left(\p_x(a_{\e_k}(f_{\e_k}))-\p_xf_{\e_k}\right)=\left(\chi_{(0,\infty)}(f_{\e_k})-1\right)\p_xf_{\e_k},$ and since
$\p_xf_{\e_k}\to\p_xf$ in $L_2(Q_T),$ there exists a function $F\in L_2(Q_T)$ such that, after possibly extracting a further subsequence,
  $|\p_xf_{\e_k}|\leq F$ almost everywhere in $Q_T$
(see the proof of Theorem 3.11 in \cite{B-R}).
We show now that  $\left(\chi_{(0,\infty)}(f_{\e_k})-1\right)\p_xf_{\e_k}\to0 $ almost everywhere in $Q_T.$
Indeed, since $\p_xf_{\e_k}\to\p_xf$ in $L_2(Q_T) ,$ we deduce that  $\p_xf_{\e_k}\to0$ almost everywhere on the set $[f=0].$
Furthermore, on the set  $[f>0]$, relation  \eqref{B:7} implies pointwise convergence $\left(\chi_{(0,\infty)}(f_{\e_k})-1\right)\to0$.
Lebesgue's dominate convergence theorem implies now the desired relation \eqref{eq:gau}, and implicitly \eqref{B:100}.

To conclude, we sum  \eqref{B:7}, \eqref{B:10}, \eqref{B:100} and let $k\to\infty$ in relation \eqref{B:B1} to obtain the first identity of Theorem \ref{T:1} $(3)$.
The corresponding relation for $g$ follows similarly.

\vspace{0.5cm}
\hspace{-0.5cm}{ \bf Acknowledgements} 
The author thanks Philippe Lauren\c cot for deep and fruitful
discussions.

\end{document}